\documentclass[10pt]{article}
\usepackage{amsmath,amssymb,amsbsy,amsfonts,amsthm,latexsym,
                        amsopn,amstext,amsxtra,euscript,amscd}
\newtheorem{theorem}{Theorem}
\newtheorem{lemma}{Lemma}

\newtheorem{corollary}{Corollary}
\newtheorem{proposition}[lemma]{Proposition}
\newtheorem{definition}{Definition}

\def\N{{\mathbb N}}

\def\Q{{\mathbb Q}}
\def\R{{\mathbb R}}

\def\\{\cr}
\def\({\left(}
\def\){\right)}
\def\[{\left[}
\def\]{\right]}
\def\<{\langle}
\def\>{\rangle}

\newcommand{\ignor}[1]{}

\def\=#1{\mathop{=}\limits_{#1}}
\def\ev#1{\mathop{ev}_{#1}}

\def\gcd{\mathop{gcd}}

\def\1{\mathbf{1}}
\def\lt{\mathrm{lt}}
\def\lm{\mathrm{lm}}

\def\Lm{\mathrm{Lm}}

\begin{document}
\title{A proof of Hilbert's Nullstellensatz based on Gr\"obner bases}
\author{L.Glebsky}
\maketitle  
\begin{abstract}
The aim of this note is to present an easy proof of Hilbert's Nullstellensatz using Gr\"obner 
basis.
I believe, that the proof has some methodical advantage in a course on Gr\"obner bases.

{\bf AMS} classification: 13-01, 13P10 
\end{abstract}
\section{Introduction and main results}

The aim of this note is to present an easy proof of the Hilbert's Nullstellensatz using 
Gr\"obner bases.
The prove presented here may be not shorter or simpler then one given in \cite{Arrondo},
however, I believe, it has some methodical advantage in a course on Gr\"obner bases.
Other proofs using Gr\"obner bases were published in \cite{Almira} and \cite{Kreuzer}. 
The proof presented in \cite{Arrondo} uses  the resultant as the main tool.
It leads to some duality between our proof and the proof of \cite{Arrondo} that will be 
explained at the end of the section.

Our proof is a sequence of propositions each of them is a good exercise on Gr\"obner bases.
As the strong Hilbert's Nullstellensatz follows from the weak one by the Rabinowitz trick, 
we prove only

\begin{theorem}[Hilbert's Nullstellensatz (weak)]
Let $k$ be an algebraically closed field. Then any nontrivial ideal 
$I\subsetneq k[x_1,\dots,x_n]$ has a solution $a\in k^n$ (that is $f(a)=0$ for any $f\in I$). 
\end{theorem} 

It turns out that for our exposition it is more natural to use Gr\"obner bases not only for 
polynomials over a field $k$ but also over a ring  $k[x]$ of polynomials in one variable. 
It allows us to consider $k[x_1,x_2,\dots,x_n]$ as $k[x_1][x_2,\dots,x_n]$ and write a short
proof for Lemma~\ref{lemma_base2}. On the other hand, $k[x]$ is an Euclidean domain, particularly,
a principle ideal domain (PID). The theory of Gr\"obner bases for polynomials over a PID
is almost the same as for polynomials over a field: one can use the same reduction process, 
Buchberger's algorithm, etc.,  see, \cite{Adams}. Particularly, it allows us to find the 
polynomial $q$ of Lemma~\ref{lemma_base2} constructively, that provides us a constructive proof
of the weak Hilbert's Nullstellensatz. In the present exposition we prove the existence
of a solution for a nontrivial ideal only and not discuss the constructivity. 
The only  facts about Gr\"obner bases we use without proof are contained in 
Proposition~\ref{prop_Groebner_basic}. Proposition~\ref{prop_Groebner_basic}
seems to be more elementary than the Buchberger's algorithm and can be proved 
using the Dickson lemma, see \cite{Cox}.  

First of all we need some notations.
Let $k$ be a field, $a\in k$. Let $\ev{a}:k[x_1,x_2,\dots,x_n]\to k[x_2,\dots,x_n]$ denote 
the evaluation homomorphism
$\ev{a}:f(x_1,x_2,\dots,x_n)\to f(a,x_2,...,x_n)$.
The proof is based on the following  lemmas.
\begin{lemma}\label{lemma_base1}
Let $k$ be an algebraically closed field,  $I\subseteq k[x_1,\dots,x_n]$ be an ideal,
such that $I\cap k[x_1]=\langle p \rangle$ and $p\in k[x]\setminus k$. Then
there exists $a\in k$, $p(a)=0$ such that $\ev{a}(I)\neq k[x_2,\dots,x_n]$.
\end{lemma}
The following lemma is valid for any field.
\begin{lemma}\label{lemma_base2}
Let $k$ be a field, $I\subseteq k[x_1,\dots,x_n]$ be an ideal,
such that $I\cap k[x_1]=\{0\}$. Then there exists a non-zero polynomial $q\in k[x_1]$
such that $\ev{a}(I)\neq k[x_2,\dots,x_n]$ for any $a\in k$, $q(a)\neq 0$.
\end{lemma}
\begin{corollary}\label{cor_base}
Let $k$ be an infinite field, $I\subseteq k[x_1,\dots,x_n]$ be an ideal,
such that $I\cap k[x_1]=\{0\}$. Then $\ev{a}(I)\neq k[x_2,\dots,x_n]$ for some
$a\in k$.
\end{corollary} 

It is clear that Lemma~\ref{lemma_base1} and Corollary~\ref{cor_base} imply the (weak)
Hilbert's Nullstellensatz by induction. The duality with the proof of \cite{Arrondo} is that 
in \cite{Arrondo} the induction goes the other direction. Precisely, in \cite{Arrondo} 
the following statement is proved. Let $I\subsetneq k[x_1,x_2,\dots,x_n]$ be an ideal.  
After some change of variables, if $(a_2,a_3,\dots,a_n)$ is a solution to 
$I\cap k[x_2,\dots,x_n]$ then $\{f(x,a_2,\dots,a_n)\;|\;f\in I\}\neq k[x]$.   

\section{Gr\" obner bases and some construction}
This section is a short introduction to Gr\"obner bases. I include it in order to make the 
exposition reasonably closed. For details one may consult \cite{Adams,Cox,Kreuzer}.

In what follows $R$ denotes a ring with unity.
An expression of the form $\alpha x_1^{k_1}x^{k_2}_2\dots x_n^{k_n}$ with $\alpha\in R$ and 
$k_i\in \N$ we call monomial. An expression of the form $x_1^{k_1}x^{k_2}_2\dots x_n^{k_n}$ or $1$
we call term. So, a polynomial in $R[x_1,x_2,\dots,x_n]$ is a sum of monomials or an 
$R$-linear combination of terms.
\begin{definition}
A total order $\preceq$ of terms is said to be a term order if $1\preceq t$ and 
$t_1\preceq t_2$ implies $tt_1\preceq tt_2$ for any terms $t,t_1,t_2$. 
\end{definition}
For example, the lexicographic order is a term order. 
Another interesting term order: let $(\alpha_1,\alpha_2,\dots\alpha_n)\in\R^n$ be 
independent over $\Q$. Then the map 
$x_1^{k_1}x_2^{k_2}\dots x_n^{k_n} \to \alpha_1k_1+\alpha_2k_2+\dots+\alpha_nk_n$ is injective and
induces a term order.
   
In what follows we assume that some term order is
fixed. Let $\lt(f)$ be a leading term of $f$ (with respect to the fixed term order).
Let $\lm(f)$ be the leading monomial of $f$ ($\lt(\lm(f))=\lt(f)$). Let 
$\Lm(I)=\{\lm(f)\;|\;f\in I\}$.

\begin{definition}
Let $I\subset R[x_1\dots x_n]$ be an ideal. 
$\Gamma\subset I\setminus \{0\}$ is called a strong Gr\"obner basis
for $I$ if for any $m\in\Lm(I)$ there exists $g\in \Gamma$ such that $\lm(g)\,|\,m$.
\end{definition} 

A Gr\"obner basis is a generating set of an ideal and has several nice properties.
\begin{proposition}\label{prop_Groebner_basic}
Let $R$ be a PID.
Then for any ideal $I\subseteq R[x_1,\dots,x_n]$ there exists a finite strong Gr\"obner basis. 
If $\Gamma$ is a strong Gr\"obner basis
for $I$ then
\begin{itemize}
\item $I=\langle\Gamma\rangle$ ($\Gamma$ generates $I$); 
\item If $R=k$ is a field then $I$ is trivial ($I=k[x_1,\dots,x_n]$) if and only if 
$\Gamma\cap k\neq\emptyset$.
\end{itemize}
\end{proposition}
Let $\phi:R_1\to R_2$ be a morphism of rings $R_1$ and $R_2$. It has the natural lift to 
the morphism $\phi:R_1[x_1,\dots,x_n]\to R_2[x_1,\dots,x_n]$. 
\begin{proposition}\label{prop_morphism}
Let $\Gamma$ be a strong Gr\"obner basis for an ideal $I\subseteq R_1[x_1,x_2,\dots,x_n]$.
Let $\phi:R_1\to R_2$ be a surjective morphism such that $\phi(a)$ neither $0$ nor a zero 
divisor 
for any $a\in \Lm(\Gamma)$. Then $\phi(\Gamma)$ is a strong Gr\"obner basis for $\phi(I)$
\end{proposition}
\begin{proof}
As $\phi$ is surjective, $\phi(I)$ is an ideal in $R_2[x_1,\dots,x_n]$. The proposition 
easily follows from

\bigskip
\noindent {\bf Statement.} For any $h\in \phi(I)$ there exists $f\in I\cap \phi^{-1}(h)$ such 
that $\lt(f)=\lt(h)$.

\bigskip
\noindent Indeed, in this case $\lm(h)$ is divisible by $\lm(\phi(g))=\phi(\lm(g))$ for a
$g\in\Gamma$ such that $\lm(g)|\lm(f)$.  So, it suffices to show the statement. We show it by
contradiction. Let $h\in\phi(I)$ contradict the statement, that is, for any $f\in I$, $h=\phi(f)$
one has $\lt(h)\neq\lt(f)$. Let $f_m$ has minimum leading term among all such $f$.
One has $\phi(\lm(f_m))=0$. On the other hand, we can eliminate $\lm(f_m)$ by some $g\in\Gamma$:
$f'=f_m-\frac{\lm(f_m)}{\lm(g)}g$. But $\phi(\frac{\lm(f_m)}{\lm(g)})=0$ ($\phi(\lm(g))$ is not a 
zero divisor). So, $\phi(f')=h$, contradiction with the  minimality. 
\end{proof}   
\section{Prove of Lemma~\ref{lemma_base1}}
\begin{proposition}\label{prop_1}
Let $k$ be a field, $f_1,f_2\in k[x_1]$, $G=\{g_1,g_2,\dots,g_r\}\subset k[x_1,x_2.\dots,x_n]$.
Let $\gcd(f_1,f_2)=1$. Then
$\langle f_1f_2,G \rangle = 
\langle f_1,G \rangle \cap 
\langle f_2,G\rangle$
\end{proposition} 
\begin{proof}
Let $Q_1,Q_2\in k[x_1]$ be such that $Q_1f_1+Q_2f_2=1$. 
We use the method $I\cap J=\langle zI, (1-z)J\rangle \cap k[x_1,\dots,x_n]$ where $z$ is a new 
variable.
Now:
$$
\langle zf_1,(z-1)f_2, zg_1,\dots,zg_r,(z-1)g_1,\dots (z-1)g_r\rangle 
\={a} \langle zf_1,(z-1)f_2, g_1,\dots,g_r \rangle \={b} 
$$
$$
\langle f_1f_2,Q_2f_2-z,g_1,\dots, g_r \rangle
$$
Equality {\bf (a)} is valid due to $g_i=zg_i-(z-1)g_i$ y $zg_i$ and $(z-1)g_i$ are multiples of 
$g_i$.
For equality {\bf (b)} it suffices to show that that $\langle zf_1, (1-z)f_2\rangle = 
\langle f_1f_2, Q_2f_2-z\rangle$.
\begin{itemize}
\item $\langle zf_1, (1-z)f_2\rangle \subseteq \langle f_1f_2, Q_2f_2-z\rangle$. Indeed, 
$zf_1=\[f_1f_2\]Q_2-\[Q_2f_2-z\]f_1$
and $(1-z)f_2=\[f_1f_2\]Q_1+\[Q_2f_2-z\]f_2$.
\item $\langle zf_1, (1-z)f_2\rangle \supseteq \langle f_1f_2, Q_2f_2-z\rangle$. Indeed, 
$f_1f_2=\[zf_1\]f_2+\[(1-z)f_2\]f_1$ and $Q_2f_2-z=\[(1-z)f_2\]Q_2-\[zf_1\]Q_1$.
\end{itemize}
Now,   $\langle f_1f_2,Q_2f_2-z,g_1,\dots, g_r \rangle\cap k[x_1,x_2,\dots,x_n]=
\langle f_1f_2,g_1,\dots, g_r \rangle$.
Indeed, let $f=pf_1f_2+p_0(Q_2f_2-z)+\sum p_ig_i$ be independent of $z$. Substituting $z=Q_2f_2$ 
we get
$f\in\langle f_1f_2,g_1,\dots, g_r \rangle$.
\end{proof}
By induction, Proposition~\ref{prop_1} implies that 
\begin{equation}\label{eq_intersec}
\langle \mathop{\Pi}_{j=1}^k(x_1-a_j)^{c_j},G\rangle =\bigcap_{j=1}^k\langle (x_1-a_j)^{c_j},G\rangle
\end{equation}
\begin{definition}
Let $I$ be an ideal. The set $\sqrt{I}=\{f\;|\;f^n\in I\; \mbox{for some }n\in\N\}$ is called 
the radical of $I$.
It is easy to check that $\sqrt{I}$ is an ideal. 
\end{definition}
\begin{proposition}
$I=k[x_1,\dots,x_n]$ if and only if $\sqrt{I}=k[x_1,\dots,x_n]$
\end{proposition}
\begin{proof}
We prove only the '$\Longleftarrow$' implication of the proposition. 
Let $1\in \sqrt{I}$. So, $1=1^n\in I$ and, consequently, $I=k[x_1,\dots,x_n]$. 
\end{proof}
\begin{corollary}
Let $k$ be an algebraically closed field, $f\in k[x_1]$, $G\subset k[x_1,\dots,x_n]$. 
Suppose, that
$\langle f,G\rangle\neq k[x_1,\dots,x_n]$. Then there exists $a\in k$, $f(a)=0$,  such that
$\langle (x_1-a),G \rangle\neq k[x_1,\dots,x_n]$.
\end{corollary}
\begin{proof}
Let $\langle f,G\rangle\neq k[x_1,\dots,x_n]$. By formula~\ref{eq_intersec} 
$\langle (x_1-a)^d,G\rangle\neq k[x_1,\dots,x_n]$ for some $a$, $f(a)=0$ and $d\in\N$. 
 Clearly,  $\langle (x_1-a),G \rangle\subset\sqrt{\langle (x_1-a)^d,G \rangle}$.
\end{proof}
Now Lemma~\ref{lemma_base1} follows due to  
$k[x_1,\dots,x_n]/\langle (x_1-a),G \rangle\sim k[x_2,\dots,x_n]/\ev{a}(\langle G \rangle)$ so,
$\langle (x_1-a),G \rangle\neq k[x_1,\dots,x_n]$ if and only
if $\ev{a}(\langle (x_1-a),G \rangle)\neq k[x_2,\dots,x_n]$. 

\section{Proof of Lemma~\ref{lemma_base2}}
Consider 
$k[x_1,\dots,x_n]$ as $k[x_1][x_2,\dots,x_n]$. So, now the polynomials has $x_2,\dots,x_n$ as 
the variables and $k[x_1]$ as a ring of coefficients. 
Let $\Gamma$ be a finite strong Gr\"obner basis for an ideal $I\subset k[x_1][x_2,\dots,x_n]$. 
Let $q\in k[x_1]$ be the product of leading coefficients of all $g\in\Gamma$.  
If $q(a)\neq 0$ then $\ev{a}(\Gamma)$ is a Gr\"obner basis of $\ev{a}(I)$ by 
Proposition~\ref{prop_morphism}. Now, $\Gamma\cap k[x_1]\subseteq I\cap k[x_1]=\emptyset$ 
and, consequently, 
$\ev{a}(\Gamma)\cap k=\emptyset$. 
The Lemma~\ref{lemma_base2} follows by Propositions~\ref{prop_Groebner_basic}. 

{\bf Acknowledgment.} The author is grateful to David Cox for pointing out an error 
in the previous version. 
The author is also grateful to J.M. Almira, L. Robbiano and E. Gordon for useful  observations.

{\it Instituto de Investigaci{\'o}n en Comunicaci\'on {\'O}ptica   
          Universidad Aut{\'o}noma de San Luis Potos{\'i} 

         Av. Karakorum 1470, Lomas 4a 78210
          San Luis Potosi, Mexico

e-mail:glebsky@cactus.iico.uaslp.mx}
\end{document}